\documentclass[11pt]{amsart}

\usepackage[T1]{fontenc}
\usepackage{lmodern,amsmath,amssymb,amsthm,mathtools,mathrsfs}
\usepackage{xcolor}
\definecolor{dblue}{rgb}{0,0,.6}
\usepackage[colorlinks=true,linkcolor=dblue,citecolor=dblue,urlcolor=dblue]{hyperref}
\allowdisplaybreaks
\numberwithin{equation}{section}

\theoremstyle{definition}
\newtheorem{Definition}{Definition}[section]
\newtheorem{Remark}[Definition]{Remark}
\theoremstyle{plain}
\newtheorem{Theorem}[Definition]{Theorem}
\newtheorem{Proposition}[Definition]{Proposition}
\newtheorem{Lemma}[Definition]{Lemma}

\newcommand{\cal}{\mathcal}
\newcommand{\Aff}{\mathbb A}

\newcommand{\QQ}{\mathbb Q}
\newcommand{\PP}{\mathbb P}
\newcommand{\HH}{\mathrm H}
\newcommand{\OO}{\cal O}
\newcommand{\calA}{\mathscr A}
\newcommand{\calY}{\mathscr Y}
\newcommand{\calB}{\mathscr B}
\newcommand{\calH}{\mathscr H}
\newcommand{\calR}{\mathscr R}

\newcommand{\rank}{\operatorname{rank}}

\newcommand{\Span}{\operatorname{Span}}

\newcommand{\Spec}{\operatorname{Spec}}
\newcommand{\length}{\operatorname{length}}
\newcommand{\GGm}{\mathbb G_m}

\title[Polynomial Compositions with Bounded Input Degree]
{Linear Independence of Polynomial Compositions with Bounded Input Degree}
\author{Xiaotian Chang}
\address{Division of Mathematical Sciences,
Nanyang Technological University}
\email{xiaotian.chang@ntu.edu.sg}

\author{Yacoub Hendi}
\address{Department of Mathematics, Uppsala University}
\email{yacoub.hendi@math.uu.se}
\hypersetup{
  pdftitle={Linear Independence of Polynomial Compositions with Bounded Input Degree},
  pdfauthor={Xiaotian Chang and Yacoub Hendi}
}
\date{}
\subjclass[2020]{14D15, 14H60, 68T05, 68T07}
\keywords{Linear independence, rational curves, deformation theory,
identifiability, neural networks}

\begin{document}

\begin{abstract}
We prove the weak composition-independence conjectures of Kohn,
Marchetti, Massarenti and Mella.  The proof uses deformations of
rational curves and a single kernel of the linearized equation.
Riemann-Roch gives a contradiction from tangent fields with
controlled poles.
By settling these conjectures, we prove that a neural network with a general polynomial activation function of sufficiently large degree is identifiable.
\end{abstract}

\maketitle

\section{Introduction}

In this paper, we consider linear relations among polynomial
compositions. Kohn, Marchetti, Massarenti and Mella,  inspired by the problem of identifiability of neural networks with general polynomial activations, ask whether,
for a general polynomial \(\sigma\) of sufficiently large degree,
the compositions \(\sigma(p_1),\ldots,\sigma(p_k)\) are linearly
independent whenever \(p_1,\ldots,p_k\) are pairwise distinct and
nonconstant \cite[Conjecture 1.2]{KMMM}. Their weak conjectures
allow the open set of composing polynomials to depend on a bound
for the degrees of the inputs. The lower bound on \(\deg\sigma\)
must depend only on \(k\) \cite[Conjectures 2.6 and 3.2]{KMMM}.

The problem is related to the theorem of Newman and Slater on linear
independence of powers of pairwise nonproportional polynomials
\cite{NS}, 
which was the main ingredient in proving identifiability for monomial neural networks \cite{FRWY24,Cra25}.
For a general composing polynomial, the inputs need only be
distinct.
The results of \cite{KMMM} for inputs of bounded degree give
a lower bound for the composition degree that also depends on that
input-degree bound. Here the latter dependence is removed.

\subsection*{Main results}
Put
\[
 A_k=\frac{5k^2-9k+8}{2},\qquad
 \calA_r=\Spec\QQ[a_0,\ldots,a_r],\qquad
 \calA_r^0=\{a_0=0\}\subset\calA_r.
\]
A point \(a\) represents the polynomial
\(\sigma_a(z)=\sum_{j=0}^r a_jz^j\).
Throughout, nonconstant means nonconstant over the coefficient field,
and linear independence is over that field.
The main purpose of this paper is to prove the following.

\begin{Theorem}\label{thm:polynomial}
Let \(k\ge2\), \(r\ge A_k\), and \(M\ge1\) be integers. There is
a nonempty \(\QQ\)-defined Zariski open subset
\(U_{k,r,M}\subset\calA_r\) with the following property.
For every extension \(L/K\) of characteristic-zero fields, every
\(\sigma\in U_{k,r,M}(K)\), every \(n\ge1\), and every ordered
tuple of pairwise distinct nonconstant polynomials
\[
 p_1,\ldots,p_k\in L[t_1,\ldots,t_n],\qquad
 \deg p_i\le M,
\]
the compositions \(\sigma(p_1),\ldots,\sigma(p_k)\) are linearly
independent over \(L\).

Moreover, \(U_{k,r,M}^0:=U_{k,r,M}\cap\calA_r^0\) is nonempty.
In particular it gives the same conclusion when both the composing
polynomial and all the inputs vanish at the origin.
\end{Theorem}

This proves Conjectures 2.6 and 3.2 of \cite{KMMM}. The open set
may depend on \(M\), whereas the bound \(A_k\) depends only
on the number of inputs. No intersection over unbounded input degrees
is taken.

We prove a rational version at the same time. If
\(p_1,\ldots,p_k\in L(t)\), the rational map
\[
 t\longmapsto[1:p_1(t):\cdots:p_k(t)]
\]
extends uniquely to a morphism \(f:\PP^1_L\to\PP^k_L\).

\begin{Definition}\label{def:joint}
The \emph{joint degree} of \((p_1,\ldots,p_k)\) is
\(\deg f^*\OO_{\PP^k}(1)\). Equivalently, write
\[
 f=[Q:P_1:\cdots:P_k]
\]
using homogeneous polynomials of one degree without a common zero. That
degree is the joint degree, and \(p_i=P_i/Q\).
\end{Definition}

For polynomial inputs this is \(\max_i\deg p_i\). For rational
inputs it measures the generated map, rather than the degrees of
separately chosen numerator--denominator pairs.

\begin{Theorem}\label{thm:rational}
The open set in Theorem~\ref{thm:polynomial} can be chosen so that,
in addition, it separates every tuple of pairwise distinct nonconstant
rational functions \(p_1,\ldots,p_k\in L(t)\) of joint degree at
most \(M\), for every extension \(L/K\) as above. Its intersection
with \(\calA_r^0\) is nonempty.
\end{Theorem}

The numerical bound is not asserted to be optimal. For instance,
\(A_2=5\), \(A_3=13\), and \(A_4=26\). The polynomial result
for two inputs in \cite[Proposition 2.1]{KMMM} has the better bound
\(r\ge3\).

Suppose that the locus of bounded-degree relations is dense.
A dominating family of relations gives a map \(f:C=\PP^1_L\to Y_a\)
and lifts of all coefficient directions. We use these lifts to
define one kernel bundle \(\calB\) on \(C\). Its defining
sequence determines \(\chi(\calB(-1))\), up to an effective
divisor whose degree is bounded by a scalar moment.

The source differential and explicit tangent fields give an exact
sequence
\[
 0\longrightarrow\OO_C(2)\oplus
   \bigoplus_{i=1}^h\OO_C(-w_iE)
   \longrightarrow\calB\longrightarrow\calR\longrightarrow0,
\]
where \(\calR\) is torsion and \(\deg E=D\ge1\).
Applying Riemann-Roch after twisting by \(\OO_C(-1)\), we obtain
\[
 0\le\length_L\calR\le(A_s-r)D-2<0,
\]
where \(s\le k\) is the number of nonzero terms in the relation.

We write \(\OO_C(B)\) for the line bundle of a divisor \(B\),
and \(\HH^i(C,\cal F)\) for sheaf cohomology. For a vector
bundle \(V\) on \(C=\PP^1_L\), Riemann-Roch gives
\begin{equation}\label{eq:RR}
 \chi(V):=\dim_L\HH^0(C,V)-\dim_L\HH^1(C,V)
       =\deg V+\rank V.
\end{equation}
In particular, \(\chi(V(-1))=\deg V\).
For a torsion sheaf \(\cal T\), we write
\[
 \length_L\cal T
   =\sum_{x\in C}\length_{\OO_{C,x}}\cal T_x\,[\kappa(x):L]
   =\chi(\cal T).
\]
We use Riemann-Roch in the form \eqref{eq:RR}
\cite[Tag 0B5B]{stacks}. The tangent and Euler sequences are
as in \cite[II, \S8]{Hartshorne}.

\subsection*{Application to identifiability of PNNs}
A fully connected polynomial neural network without biases, of widths
$\mathbf{d} = (d_0, \dots, d_L)$, is the map
\[
  f_W = W_L \circ \sigma \circ W_{L-1} \circ \sigma \circ \cdots \circ \sigma \circ W_1,
\]
where $W_i \in \mathbb{R}^{d_i \times d_{i-1}}$ and $\sigma$ is a univariate polynomial
applied componentwise. Hence each weight tuple $W = (W_1, \dots, W_L)$ defines a network,
and permuting the neurons of a hidden layer does not change $f_W$. Similarly, removing
inactive neurons (those whose incoming or outgoing weights all vanish) or merging redundant
ones (two neurons in the same layer with equal nonzero incoming weights) does not change $f_W$.
We call $W$ a \emph{subnetwork} if it has at least one inactive or redundant neuron, and we
call $W$ \emph{identifiable} if every $V$ with $f_V = f_W$ is obtained from $W$ by permuting
the neurons of the hidden layers.
In~\cite[Corollary~4.1]{KMMM} it is shown, based on the origin-passing version of the
conjecture ~\cite[Conjecture 3.2]{KMMM}, that for $\sigma$ with sufficiently large degree these three operations are the only ways to
change the network without changing the map $f_W$. Therefore, proving
Theorem~\ref{thm:polynomial} immediately yields the following theorem.

\begin{Theorem}[Identifiability of bias-free PNNs for generic $\sigma$]\label{thm:identifiability}
Let $\mathbf{d} = (d_0, \dots, d_L)$ with $L \ge 2$, let $D = \max(d_0, \dots, d_{L-1})$, and let
\[
  r \ \ge\ A_{2D} \ =\ 10 D^2 -9D + 4.
\]
There is a nonempty Zariski open set $U \subseteq \mathbb{R}^0[u]_{\le r}$ such that, for every
$\sigma \in U$, the weights $W$ in architecture $\mathbf{d}$ are identifiable if and only if
$W$ is not a subnetwork. 
Moreover if the architecture that has layer-specific activations $\sigma_i$ with $r_i := \deg \sigma_i$ for $1\leq i \leq L-1$ then the same holds if $r_i \geq A_{2d_i}$.
In particular, for each $i$ there is a nonempty open zariski sets $U_i \subset \mathbb{R}^0[u]_{\leq r_i}$ such that for any choice $(\sigma_1, \dots, \sigma_{L-1}) \in U_1 \times \dots \times U_{L-1}$, the weights $W$ are identifiable if and only if $W$ is not subnetwork.
\end{Theorem}

\section{Families of Relations}

Let
\[
 \Sigma_r=\{a_r\ne0,\ \operatorname{disc}(\sigma_a)\ne0\}
            \subset\calA_r.
\]
Here, if \(\sigma_a(z)=a_r\prod_{i=1}^r(z-\alpha_i)\) over
an algebraic closure, its discriminant is
\[
 \operatorname{disc}(\sigma_a)
    =a_r^{2r-2}\prod_{i<j}(\alpha_i-\alpha_j)^2.
\]
It is a polynomial in the coefficients \(a_j\). Thus \(\Sigma_r\)
is the open set of squarefree polynomials of exact degree \(r\).
We first construct the family of bounded-degree relations.

\begin{Lemma}\label{lem:incidence}
Fix \(k,r,M\). Call a geometric point of \(\Sigma_r\) bad if
it admits a nontrivial scalar relation among \(k\) compositions.
The inputs are required to be distinct nonconstant rational
functions of joint degree at most \(M\). The bad locus is the
image of a morphism \(\mathscr I\to\Sigma_r\) of finite-type
\(\QQ\)-schemes.
\end{Lemma}

\begin{proof}
In a nonzero relation, delete the terms with zero coefficient and
divide by one of the remaining coefficients. We may then assume
that the relation has \(s\) terms, with all \(\lambda_i\ne0\)
and \(\lambda_s=1\). A one-term relation is impossible, since
composition of a nonconstant rational function with a nonconstant
polynomial is nonconstant. Thus \(2\le s\le k\).
After deleting coordinates, cancel the common homogeneous factor
in \((Q,P_1,\ldots,P_s)\). The resulting joint degree is at least
one and at most the original joint degree.

For \(2\le s\le k\) and \(1\le D\le M\), let
\(Q,P_1,\ldots,P_s\) be homogeneous polynomials of degree \(D\)
in two variables, without a common zero, taken up to a common
nonzero scalar. Let \(Q\ne0\), let no \(P_i\) be proportional
to \(Q\), and let \(P_i\ne P_j\) for \(i\ne j\).
These conditions define an open subset of a projective space.
Indeed, the common-zero locus is the proper image of a closed
incidence over \(\PP^1\); the other excluded conditions are
given by coefficient equalities or minors.

Adjoin \((\lambda_1,\ldots,\lambda_{s-1})\in(\GGm)^{s-1}\),
put \(\lambda_s=1\), and impose the coefficient equations of
\[
 \sum_{i=1}^s\lambda_i
       \sum_{j=0}^r a_jP_i^jQ^{r-j}=0.
\]
The finite disjoint union of these incidences over \(s,D\) is a
finite-type \(\QQ\)-scheme \(\mathscr I\). Its image is the
bad locus in the statement. To see the reverse inclusion when
\(s<k\), complete the tuple by distinct translates \(p_1+c\)
and assign zero coefficients to the added terms. Over a geometric
field this is possible, and the joint degree remains \(D\).
The image is constructible by Chevalley's theorem
\cite[Tag 054K]{stacks}.
\end{proof}

Put \(\Lambda_s=(\GGm)^{s-1}\), with \(\lambda_s=1\), and set
\[
 F_a(X,Z)=\sum_{j=0}^r a_jX^jZ^{r-j}.
\]
Consider the universal incidence
\begin{equation}\label{eq:universal}
 \calY_s=
 \left\{\sum_{i=1}^s\lambda_iF_a(X_i,Z)=0\right\}
 \subset\PP^s\times\Lambda_s\times\Sigma_r.
\end{equation}
Write \(\pi:\calY_s\to\Sigma_r\) for the projection and
\(Y_a\) for its fiber over \(a\). The order of
homogeneous coordinates will be written as \([Z:X_1:\cdots:X_s]\).

\begin{Lemma}\label{lem:smooth-incidence}
The morphism \(\calY_s\to\Sigma_r\) is smooth. The projection
\(Y_a\to\Lambda_s\) is smooth along \(Z=0\).
\end{Lemma}

\begin{proof}
On \(Z=1\), the defining equation is
\[
 \Phi(a,x,\lambda)=\sum_i\lambda_i\sigma_a(x_i)=0.
\]
Its spatial derivatives are \(\lambda_i\sigma_a'(x_i)\), and
its derivatives in the scalar coordinates are \(\sigma_a(x_j)\)
for \(j<s\). If all these derivatives vanished, then
\(\sigma_a(x_1)=\sigma_a'(x_1)=0\). This contradicts squarefreeness.
Here \(s\ge2\), so the scalar coordinate \(\lambda_1\) is
available.

At \(Z=0\), the spatial homogeneous derivatives include
\[
 ra_r\lambda_iX_i^{r-1},\qquad 1\le i\le s.
\]
At least one is nonzero. Euler's identity shows that the homogeneous
differential annihilates the radial direction on the hypersurface;
it therefore gives a nonzero cotangent vector on \(\PP^s\).
The relative spatial differential is consequently surjective there.
These computations prove both assertions by the relative Jacobian
criterion.
\end{proof}

Suppose that the bad locus is dense. After base change to an
algebraically closed field \(L\) of characteristic zero, choose
a reduced irreducible component \(\calH\) of \(\mathscr I_L\)
which dominates \(\Sigma_{r,L}\). The number \(s\) of active
inputs and their joint degree \(D\) are fixed on \(\calH\).
Write \(q:\calH\to\Sigma_{r,L}\) for the projection.
Thus \(\calH\) parametrizes relations with fixed source \(\PP^1_L\),
fixed joint degree \(D\), and varying composing polynomial and scalars;
the map \(q\) remembers only the composing polynomial.

The smooth locus of \(q\) is nonempty. Indeed, generic flatness
and the existence of a smooth dense open in its integral generic
fiber give a smooth point of \(q\). Here the generic residue
field has characteristic zero and is therefore perfect
\cite[Tags 052B, 04QM and 01V4]{stacks}.
Choose \(b\in\calH(L)\) in this smooth locus and put \(a=q(b)\).
Then \(dq_b:T_b\calH\to T_a\Sigma_r\) is surjective.
We now omit base-change subscripts and take tangent bundles over \(L\).
The corresponding relation is
\begin{equation}\label{eq:relation}
 \sum_{i=1}^s\lambda_i\sigma_a(p_i)=0,
 \qquad\lambda_i\ne0,\quad\lambda_s=1,
 \qquad 2\le s\le k.
\end{equation}
Its generated coordinates give
\[
 f:C=\PP^1_L\longrightarrow Y_a,
 \qquad f=([Q:P_1:\cdots:P_s],\lambda),
 \qquad f^*\OO_{\PP^s}(1)=\OO_C(D).
\]
Let
\begin{equation}\label{eq:pole-divisor}
 E=\operatorname{div}(Q),\qquad\deg E=D.
\end{equation}
Every \(p_i=P_i/Q\) has pole divisor at most \(E\).
For polynomial inputs of maximum degree \(D\), this is
\(E=D[\infty]\).

\begin{Lemma}\label{lem:horizontal}
The sequence of global sections
\[
 0\longrightarrow\HH^0(C,f^*T_{Y_a})
 \longrightarrow\HH^0(C,f^*T_{\calY_s})
 \xrightarrow{d\pi}T_a\Sigma_r\longrightarrow0
\]
is exact. Any \(L\)-linear section of its last map induces a
splitting of the pulled-back tangent sequence
\begin{equation}\label{eq:horizontal-sequence}
 0\longrightarrow f^*T_{Y_a}
 \longrightarrow f^*T_{\calY_s}
 \xrightarrow{d\pi}\OO_C^{r+1}\longrightarrow0.
\end{equation}
\end{Lemma}

\begin{proof}
The sequence \eqref{eq:horizontal-sequence} is exact by
Lemma~\ref{lem:smooth-incidence}. Taking global sections is left
exact, and \(\HH^0(C,\OO_C)=L\). It remains to prove that
\(d\pi:\HH^0(C,f^*T_{\calY_s})\to T_a\Sigma_r\) is surjective.

After shrinking \(\calH\) around \(b\), we may assume that
\(\calH\) is smooth. The universal homogeneous coordinates
define an evaluation morphism
\[
 \operatorname{ev}:C\times\calH\longrightarrow\calY_s,
 \qquad\pi\circ\operatorname{ev}=q\circ\operatorname{pr}_{\calH},
\]
whose restriction to \(C\times\{b\}\) is \(f\).
Indeed, the coordinates have no common zero, and their common
scalar ambiguity does not affect the projective morphism.

Differentiating this family at \(b\), with \(C\) fixed, gives
\[
 T_b\calH\longrightarrow\HH^0(C,f^*T_{\calY_s}),
\]
whose composition with \(d\pi\) is \(dq_b\). Since \(dq_b\)
is surjective, so is \(d\pi\) on global sections.

Choose an \(L\)-linear section
\[
 H_0:T_a\Sigma_r\longrightarrow\HH^0(C,f^*T_{\calY_s}).
\]
Extending \(H_0\) by \(\OO_C\)-linearity gives
\[
 H:T_a\Sigma_r\otimes_L\OO_C\longrightarrow f^*T_{\calY_s},
 \qquad c\otimes g\longmapsto gH_0(c).
\]
Then \(d\pi\circ H=\operatorname{id}\), so \(H\) splits
\eqref{eq:horizontal-sequence}. The splitting is regular on all
of \(C\), since the images of \(H_0\) are global sections.
\end{proof}

The two spaces of global sections in Lemma~\ref{lem:horizontal}
parametrize first-order deformations of \(f\), with \(C\) fixed,
in \(Y_a\) and \(\calY_s\), respectively. Their quotient
\(T_a\Sigma_r\) records the variation of the composing polynomial.
Fix a section \(H_0\) as above, its induced splitting \(H\), and
the projection
\[
 P=1-H\circ d\pi:f^*T_{\calY_s}\longrightarrow f^*T_{Y_a}.
\]
\section{The Linearized Equation}

Put \(X=\PP^s\times\Lambda_s\), and let \(j:Y_a\hookrightarrow X\)
be the inclusion. By Lemma~\ref{lem:smooth-incidence}, \(Y_a\)
is a smooth hypersurface with normal bundle
\(N_{Y_a/X}\simeq\OO_{Y_a}(r)\), where the twist is taken
from \(\PP^s\). Dualizing the conormal sequence gives
\cite[Tag 06BB]{stacks}
\[
 0\longrightarrow T_{Y_a}\longrightarrow j^*T_X
 \xrightarrow{d\Phi_a}\OO_{Y_a}(r)\longrightarrow0.
\]
The normal map \(d\Phi_a\) is the differential of the defining
equation. Pullback along \(f\) preserves exactness, since the
quotient is locally free. The scalar component of \(f\) is the
constant \(\lambda\), and \(f^*\OO_{\PP^s}(1)=\OO_C(D)\).
We therefore obtain
\begin{equation}\label{eq:normal-pullback}
 0\longrightarrow f^*T_{Y_a}
 \longrightarrow f^*T_{\PP^s}\oplus
       (T_\lambda\Lambda_s\otimes_L\OO_C)
 \xrightarrow{f^*d\Phi_a}\OO_C(rD)\longrightarrow0.
\end{equation}

The scalar projection induces
\[
 \operatorname{pr}_\Lambda:\HH^0(C,f^*T_{\calY_s})
       \longrightarrow T_\lambda\Lambda_s\simeq L^{s-1}.
\]
Let \(W\) be the scalar image of the chosen lifting \(H_0\):
\begin{equation}\label{eq:W}
 W=\operatorname{im}(\operatorname{pr}_\Lambda\circ H_0),
 \qquad\ell=\dim W.
\end{equation}
For the coefficient basis \(e_0,\ldots,e_r\), put
\[
 v_j=\operatorname{pr}_\Lambda H_0(e_j),
 \qquad W=\Span_L\{v_0,\ldots,v_r\}.
\]
These are the constant scalar components of \(H(e_j)\).
The space \(W\) depends on the choice of \(H_0\): the scalar
projection need not vanish on \(\HH^0(C,f^*T_{Y_a})\), and
therefore need not descend to the quotient \(T_a\Sigma_r\).

Restrict the last map of \eqref{eq:normal-pullback} to obtain
\[
 \mu_W:f^*T_{\PP^s}\oplus(W\otimes_L\OO_C)
       \longrightarrow\OO_C(rD),
 \qquad\calB=\ker\mu_W.
\]
Taking its kernel gives, as subsheaves of \(f^*j^*T_X\),
\[
 \calB=f^*T_{Y_a}\cap
       \bigl(f^*T_{\PP^s}\oplus(W\otimes_L\OO_C)\bigr).
\]
Equivalently,
\[
 \calB=\ker\!\left(f^*T_{Y_a}\longrightarrow
       (T_\lambda\Lambda_s/W)\otimes_L\OO_C\right),
\]
where the map is induced by scalar projection. Thus \(\calB\)
consists of tangent sections in \(Y_a\) whose scalar components
lie in \(W\otimes_L\OO_C\).
On \(Q\ne0\), the restricted normal map \(\mu_W\) is
\[
 (u,w)\longmapsto
 \sum_{i=1}^s\lambda_i\sigma_a'(p_i)u_i
       +\sum_{i<s}\sigma_a(p_i)w_i.
\]
The full normal map in \eqref{eq:normal-pullback} is surjective,
but its restriction \(\mu_W\) need not be. The latter has generic
rank one, since the \(p_i\) are nonconstant and the characteristic
is zero. Hence its image is \(\OO_C(rD)(-\Delta)\) for an
effective divisor \(\Delta\), which records this failure of
surjectivity.
Thus \(\calB\) is a vector bundle and
\begin{equation}\label{eq:kernel-sequence}
 0\longrightarrow\calB
 \longrightarrow f^*T_{\PP^s}\oplus(W\otimes_L\OO_C)
 \xrightarrow{\mu_W}\OO_C(rD)(-\Delta)\longrightarrow0.
\end{equation}
Write
\[
 \tau=\deg\Delta,\qquad h=\rank\calB-1=s+\ell-2\le2s-3.
\]
The pulled-back Euler sequence is
\[
 0\longrightarrow\OO_C\longrightarrow\OO_C(D)^{s+1}
   \longrightarrow f^*T_{\PP^s}\longrightarrow0.
\]
Twisting this sequence and \eqref{eq:kernel-sequence} by
\(\OO_C(-1)\), Riemann-Roch gives
\begin{equation}\label{eq:kernel-chi}
 \chi(\calB(-1))=(s+1-r)D+\tau.
\end{equation}

\begin{Lemma}\label{lem:defect}
We have \(\tau\le(s-1)D\). If \(\ell=s-1\), then \(\Delta=0\).
\end{Lemma}

\begin{proof}
By Lemma~\ref{lem:smooth-incidence}, the spatial differential
alone is surjective at target infinity. Hence \(\Delta\)
is disjoint from \(E\).

Write \(H(e_j)=(u_j,v_j,e_j)\) in the ambient tangent bundle.
Tangency to the universal equation gives
\begin{equation}\label{eq:moment-lift}
 \mu_W(u_j,v_j)=-M_j,
 \qquad M_j=\sum_i\lambda_iP_i^jQ^{r-j}.
\end{equation}
Thus every \(M_j\in\HH^0(C,\OO_C(rD))\) vanishes along
\(\Delta\), with its full multiplicities.
The Vandermonde matrix \((p_i^j)_{0\le j<s,1\le i\le s}\)
is invertible over \(L(t)\). Since \((\lambda_i)\ne0\),
some moment \(\sum_i\lambda_ip_i^\nu\) is nonzero for
\(0\le\nu\le s-1\). Its numerator
\[
 \widehat S_\nu=\sum_i\lambda_iP_i^\nu
       \in\HH^0(C,\OO_C(\nu D))
\]
is a nonzero section. Since \(M_\nu=Q^{r-\nu}\widehat S_\nu\)
and \(\Delta\) is disjoint from \(E\), we obtain
\[
 \Delta\le\operatorname{div}(\widehat S_\nu),
 \qquad\tau\le\nu D\le(s-1)D.
\]
Here we use \(r\ge s-1\), so that \eqref{eq:moment-lift}
is available for \(j=\nu\).

If \(\ell=s-1\), then \(W=L^{s-1}\) and \(\mu_W\) is
the full normal map in \eqref{eq:normal-pullback}. It is
surjective, so \(\Delta=0\).
\end{proof}

\section{Tangent Fields and Riemann-Roch}

We construct explicit tangent fields on the affine part \(Z=1\)
of \(\calY_s\), all with zero scalar component. Write a coefficient
variation as a polynomial \(\delta\sigma(z)\) of degree at most
\(r\), and put
\[
 T_x(z)=\prod_{i=1}^s(z-x_i)
       =z^s+b_1(x)z^{s-1}+\cdots+b_s(x),
 \qquad
 S_j(x,\lambda)=\sum_i\lambda_ix_i^j.
\]
The degree in the variables \(x_1,\ldots,x_s\) of \(b_j\) is
\(j\). Parameters \(a\) and \(\lambda\) have degree zero in
all the bounds below.

\begin{Lemma}\label{lem:frame}
Assume \(r\ge s\). On the locus of distinct affine coordinates,
the following fields span the tangent space to
\(\calY_s\to\Lambda_s\):
\begin{alignat}{3}
 &\delta\sigma=z^jT_x,&\qquad &\delta x_i=0,
       &\quad&0\le j\le r-s;\label{eq:frame-A}\\
 &\delta\sigma=S_i z^j-S_jz^i,&&\delta x_l=0,
       &&0\le i<j\le s-1;\label{eq:frame-B}\\
 &\delta\sigma=R_m:=-z^m\sigma_a'(z)+T_xQ_m,
       &&\delta x_i=x_i^m,
       &&0\le m\le s-1.\label{eq:frame-C}
\end{alignat}
Here \(Q_0=Q_1=0\), and for \(m\ge2\), \(Q_m\) cancels only
the terms of \(-z^m\sigma_a'\) of degree strictly greater than
\(r\). Each coefficient of the first two families has spatial
degree at most \(s\). The same is true for
\eqref{eq:frame-C} when \(m\le2\); for \(m\ge3\), its
coefficient degrees are at most \(s+m-2\).
\end{Lemma}

\begin{proof}
For \eqref{eq:frame-A}, every weighted evaluation at the nodes
vanishes. For \eqref{eq:frame-B}, it is
\(S_iS_j-S_jS_i=0\). Their degree bounds follow immediately.

For \(m\ge2\), write
\begin{equation}\label{eq:truncated-quotient}
 Q_m(z)=\sum_{u=0}^{m-2}q_u z^{r+m-1-s-u}.
\end{equation}
The smallest exponent is \(r+1-s\ge0\). Top-term cancellation
determines its coefficients successively by
\begin{equation}\label{eq:recurrence}
 q_u=(r-u)a_{r-u}-\sum_{j=1}^u b_jq_{u-j},
                  \qquad 0\le u\le m-2.
\end{equation}
Induction gives \(\deg_xq_u\le u\). The product \(b_jq_u\)
has degree at most \(j+u\), so every coefficient of \(R_m\)
has degree at most \(s+m-2\). The spatial coefficients have
degree \(m\), which satisfies the same bound. For \(m=0,1\),
the coefficients of \(R_m\) have spatial degree zero and the spatial
coefficients have degree at most one. In every case
\(\deg_zR_m\le r\).

Tangency of \eqref{eq:frame-C} is the polynomial identity
\[
 \delta\Phi
  =\sum_i\lambda_i\bigl(\sigma_a'(x_i)x_i^m+R_m(x_i)\bigr)
  =\sum_i\lambda_iT_x(x_i)Q_m(x_i)=0.
\]
No inverse of a difference of nodes occurs in these fields.

To prove spanning, first prescribe a spatial variation. The vectors
\((x_i^m)_{i=1}^s\), \(0\le m\le s-1\), form a basis at
distinct nodes by the Vandermonde determinant. Subtract a combination
of \eqref{eq:frame-C} to leave zero spatial variation. The remaining
coefficient variation \(A(z)\) satisfies
\(\sum_i\lambda_iA(x_i)=0\). Divide \(A\) by the monic
polynomial \(T_x\). Its quotient is accounted for by
\eqref{eq:frame-A}, and its remainder has degree less than \(s\).

On this remainder space, the weighted-evaluation functional has
coordinates \((S_0,\ldots,S_{s-1})\). They are not all zero,
again by Vandermonde and the nonzero scalar vector. If \(S_\nu\ne0\),
the vectors \(S_\nu e_j-S_je_\nu\), \(j\ne\nu\), span its
kernel. These are among \eqref{eq:frame-B}, up to sign.
This proves spanning of the entire zero-scalar tangent space.
\end{proof}

Assign to every field in the first two families the pole bound \(s\).
Assign the same pole bound to \eqref{eq:frame-C} with \(m\le2\),
and assign \(s+m-2\) to its field with index \(m\ge3\).
There is only one field with each such exceptional index. If
\[
 e_s=\sum_{m=3}^{s-1}(m-2)=\frac{(s-3)(s-2)}2,
\]
where the sum is empty for \(s=2,3\), then any \(h\) distinct
fields have total pole bound at most \(sh+e_s\).

\begin{Proposition}\label{prop:field-sequence}
There are \(h\) tangent fields with pole bounds
\(w_1,\ldots,w_h\), where \(\sum_iw_i\le sh+e_s\),
which give an exact sequence
\begin{equation}\label{eq:field-sequence}
 0\longrightarrow\OO_C(2)\oplus
   \bigoplus_{i=1}^h\OO_C(-w_iE)
   \longrightarrow\calB\longrightarrow\calR\longrightarrow0
\end{equation}
with \(\calR\) torsion. Consequently,
\begin{equation}\label{eq:RR-length}
 0\le\length_L\calR
   =\left(s+1-r+\sum_iw_i\right)D+\tau-2.
\end{equation}
\end{Proposition}

\begin{proof}
The affine translation field \(\partial/\partial x_i\) extends
to a global vector field on \(\PP^s\), induced by
\(X_i\mapsto X_i+\varepsilon Z\). The coefficient directions
are global on \(\Sigma_r\). A polynomial of spatial degree
at most \(w\), evaluated at the \(p_i\), has pole divisor
at most \(wE\). Each field in Lemma~\ref{lem:frame} therefore
pulls back to a regular section of the ambient tangent bundle
twisted by \(wE\). Its image under the hypersurface
differential vanishes away from \(E\), hence everywhere.
It is thus a section of \(f^*T_{\calY_s}(wE)\).

Apply the regular projection \(P\) defined after
Lemma~\ref{lem:horizontal}. The original field has zero scalar
component, and its horizontal part has scalar component in
\(W\otimes_L\OO_C\). Its projection consequently lies in
\(\HH^0(C,\calB(wE))\).

These projected sections span \(\calB\) at the generic point
\(\eta\) of \(C\). Indeed, for \(v\in\calB_\eta\), choose a
coefficient vector \(c\in L(t)^{r+1}\) such that \(H(c)\) has the opposite scalar
component. Then \(v+H(c)\) has zero scalar component.
Lemma~\ref{lem:frame} expresses it as a combination of the
displayed fields, since the \(p_i\) are distinct. Applying
\(P\) recovers \(v\).

The source differential has zero scalar component and gives
\[
 df:\OO_C(2)\longrightarrow\calB.
\]
It is generically nonzero because \(f\) is nonconstant and
the characteristic is zero. Choose \(h\) projected fields
which, together with \(df\), form a basis at the generic point.
The resulting map in \eqref{eq:field-sequence} is generically
an isomorphism. Its kernel is zero because its source is locally
free, and its cokernel \(\calR\) is torsion. This also covers
\(h=0\), when only \(df\) is needed. The chosen fields satisfy
\(\sum_iw_i\le sh+e_s\).

Twist \eqref{eq:field-sequence} by \(\OO_C(-1)\).
Riemann-Roch and \eqref{eq:kernel-chi} give
\begin{align*}
 \length_L\calR
 &=\chi(\calB(-1))-2+D\sum_iw_i\\
 &=\left(s+1-r+\sum_iw_i\right)D+\tau-2.
\end{align*}
Twisting preserves the length of a torsion sheaf, so this proves
\eqref{eq:RR-length}.
\end{proof}

\section{Proof of the Theorems}

\begin{proof}[Proof of Theorem~\ref{thm:rational}]
Suppose that the bounded-joint-degree bad locus is dense.
Section 2 supplies a relation in a dominating family, with
\(2\le s\le k\) and \(D\ge1\). The sequence \(A_s\) is
increasing, since \(A_{s+1}-A_s=5s-2>0\), and \(A_s\ge s\).
Thus \(r\ge A_k\) meets the hypotheses of the preceding lemmas.

If \(\ell=s-1\), then \(h=2s-3\) and \(\tau=0\) by
Lemma~\ref{lem:defect}. Proposition~\ref{prop:field-sequence}
gives
\begin{align*}
 0\le\length_L\calR
 &\le\bigl(s+1-r+s(2s-3)+e_s\bigr)D-2\\
 &=(A_s-r)D-2<0.
\end{align*}
If \(\ell\le s-2\), then \(h\le2s-4\) and
\(\tau\le(s-1)D\). The same proposition gives
\begin{align*}
 0\le\length_L\calR
 &\le\bigl(2s-r+s(2s-4)+e_s\bigr)D-2\\
 &=(A_s-1-r)D-2<0.
\end{align*}
Both cases are impossible. The closure of the bad image is
therefore proper in \(\Sigma_r\).
Let \(Z_{k,r,M}\) be the scheme-theoretic image of
\(\mathscr I\to\Sigma_r\).
Its support is this closure, since the morphism is of finite type
\cite[Tag 01R5]{stacks}. Define
\[
 U_{k,r,M}=\Sigma_r\setminus Z_{k,r,M}.
\]
The incidence and its scheme-theoretic image are defined over
\(\QQ\), so \(U_{k,r,M}\) is a nonempty open in \(\calA_r\).
Moreover,
\begin{equation}\label{eq:empty-incidence}
 \mathscr I\times_{\Sigma_r}U_{k,r,M}=\varnothing.
\end{equation}

Let \(K\) be a field of characteristic zero. Its prime subfield
is \(\QQ\), and
\[
 \calA_{r,K}:=\calA_r\times_{\Spec\QQ}\Spec K
       =\Spec K[a_0,\ldots,a_r].
\]
The base change \(U_{k,r,M,K}\subset\calA_{r,K}\) is open
\cite[Tag 01JY]{stacks}. More explicitly, a principal open
\(D(g)\subset\calA_r\) becomes the principal open defined by
the same polynomial \(g\) over \(K\), since
\[
 \QQ[a_0,\ldots,a_r]_g\otimes_{\QQ}K
       \cong K[a_0,\ldots,a_r]_g.
\]
The nonempty open \(U_{k,r,M}\) contains such a \(D(g)\) with
\(g\ne0\). Its image in \(K[a_0,\ldots,a_r]\) is nonzero,
so \(U_{k,r,M,K}\) is nonempty. It has \(K\)-points because
\(K\) is infinite.

For any extension \(L/K\), base change of
\eqref{eq:empty-incidence} gives an empty incidence over
\(U_{k,r,M,L}\). A relation over \(L\) with a composing polynomial
in \(U_{k,r,M}(K)\) would give an \(L\)-point of this
incidence. This proves the asserted exclusion over every such
extension.

It remains to show that the open meets \(\calA_r^0\).
For any constant \(c\), let
\[
 (T_c\sigma)(z)=\sigma(z+c).
\]
The identity
\[
 \sum_i\lambda_i\sigma(p_i)=0
 \quad\Longleftrightarrow\quad
 \sum_i\lambda_i(T_c\sigma)(p_i-c)=0
\]
shows that the geometric bad locus is translation invariant.
The input shift preserves joint degree: it changes generated
coordinates by the invertible linear transformation
\[
 [Q:P_1:\cdots:P_k]
   \longmapsto[Q:P_1-cQ:\cdots:P_k-cQ].
\]
It also preserves distinctness and nonconstancy. Translations
preserve \(\Sigma_r\), so its bad-image closure and its good
complement are invariant.

Choose a geometric point \(\sigma\in U_{k,r,M}\), and choose a
root \(\alpha\) of \(\sigma\) over an algebraic closure.
Then \(T_\alpha\sigma\in U_{k,r,M}\) and
\((T_\alpha\sigma)(0)=0\). Consequently
\(U_{k,r,M}\cap\calA_r^0\) is a nonempty relative open
defined over \(\QQ\). This assertion places no origin condition
on the rational inputs. Since \(\calA_r^0\cong\Aff^r_{\QQ}\),
the same principal-open argument shows that this intersection
has a nonempty base change, with rational points, over every
characteristic-zero coefficient field.
\end{proof}

\begin{proof}[Proof of Theorem~\ref{thm:polynomial}]
For univariate polynomial inputs of degrees at most \(M\), the
generated map has joint degree at most \(M\). Thus
Theorem~\ref{thm:rational} gives the assertion over every field
extension, and also the nonempty origin slice.

For multivariate inputs, choose a vector \(v\in L^n\) such that
restriction along \(t\mapsto tv\) preserves the positive degrees
of every \(p_i\), and preserves every nonzero difference
\(p_i-p_j\). Such vectors form a nonempty open: for each
nonconstant polynomial or nonconstant difference, it suffices to
avoid the zero set of its highest nonzero homogeneous part; a
nonzero constant difference imposes no condition. The field \(L\)
is infinite. Restriction therefore gives distinct nonconstant
univariate inputs of degree at most \(M\). Any scalar relation
among the original compositions restricts to the same scalar
relation among these compositions, which is impossible.

This line passes through the origin, so it also preserves
\(p_i(0)=0\). Moreover, the ordinary separation property on
\(U_{k,r,M}\cap\calA_r^0\) already applies to such inputs;
there is no need to preserve their origin values during the
translation argument for the composing polynomial.
\end{proof}

\begin{Remark}
The proof bounds the degree threshold independently of \(M\), but
does not bound the degree of a relation at an individual exceptional
composing polynomial. Properness of each closed set
\(Z_{k,r,M}\) does not imply properness of the closure of their
union. The distinction between an open for each finite \(M\)
and one open for all input degrees is therefore retained.
\end{Remark}

\subsection*{Acknowledgments.}
We thank Kathlén Kohn and Giovanni Luca Marchetti for helpful discussions
on the conjectures of \cite{KMMM} and their connection to the
identifiability of polynomial neural networks. 
X.C. thanks his Ph.D. supervisor, Prof.~Xi Chen, for teaching him algebraic
geometry, and Prof.~Cyrus Mostajeran for his advice and support in research.
Y.H. is grateful to the
organizers of the Wallenberg Advanced Scientific Forum Fall 2026 in Stockholm, where
part of this work was discussed, and to his supervisor Magdalena Larfors
for her guidance and support.
Y.H.'s work was supported by the Wallenberg
AI, Autonomous Systems and Software Program (WASP) funded by the Knut and
Alice Wallenberg Foundation.

\subsection*{Acknowledgment of AI assistance}
We used the large language model Claude (Anthropic) for discussion of the mathematical arguments, for symbolic verification of computations, and for assistance with writing and \LaTeX{} typesetting. All results and proofs were verified by the authors, who are solely responsible for the content of this paper.


\begin{thebibliography}{KMMM26}

\bibitem[Har77]{Hartshorne}
R. Hartshorne, \emph{Algebraic Geometry},
Graduate Texts in Mathematics, vol.~52, Springer-Verlag, New York, 1977.
\href{https://doi.org/10.1007/978-1-4757-3849-0}{doi:10.1007/978-1-4757-3849-0}.

\bibitem[KMMM26]{KMMM}
K. Kohn, G. L. Marchetti, A. Massarenti and M. Mella,
\emph{Linear independence of polynomial compositions and identifiability of deep neural networks},
2026, \href{https://arxiv.org/abs/2608.27113v1}{arXiv:2608.27113v1}.

\bibitem[NS79]{NS}
D. J. Newman and M. Slater,
\emph{Waring's problem for the ring of polynomials},
J. Number Theory \textbf{11} (1979), no.~4, 477--487.
\href{https://doi.org/10.1016/0022-314X(79)90027-1}{doi:10.1016/0022-314X(79)90027-1}.

\bibitem[FRWY24]{FRWY24}
B. Finkel, J. I. Rodriguez, C. Wu and T. Yahl,
\emph{Activation degree thresholds and expressiveness of polynomial neural networks},
2024, \href{https://arxiv.org/abs/2408.04569v4}{arXiv:2408.04569v4}.

\bibitem[Cra25]{Cra25}
A. Cr\u{a}ciun,
\emph{Linear independence of powers for polynomials},
2025, \href{https://arxiv.org/abs/2507.10163v1}{arXiv:2507.10163v1}.

\bibitem[Stacks]{stacks}
The Stacks Project Authors, \emph{The Stacks Project},
\url{https://stacks.math.columbia.edu}, Tags
\href{https://stacks.math.columbia.edu/tag/054K}{054K},
\href{https://stacks.math.columbia.edu/tag/0B5B}{0B5B},
\href{https://stacks.math.columbia.edu/tag/052B}{052B},
\href{https://stacks.math.columbia.edu/tag/04QM}{04QM},
\href{https://stacks.math.columbia.edu/tag/01V4}{01V4},
\href{https://stacks.math.columbia.edu/tag/01R5}{01R5},
\href{https://stacks.math.columbia.edu/tag/01JY}{01JY}, and
\href{https://stacks.math.columbia.edu/tag/06BB}{06BB},
accessed October 6, 2026.
\end{thebibliography}
\end{document}